\documentclass{amsart}
\usepackage{amssymb}
\usepackage{amsfonts}

\newtheorem{theorem}{Theorem}
\theoremstyle{plain}

\newtheorem{corollary}{Corollary}

\newtheorem{definition}{Definition}

\newtheorem{lemma}{Lemma}

\newtheorem{proposition}{Proposition}

\numberwithin{equation}{section}
\input{tcilatex}

\begin{document}
\title[Hadamard-type inequalities]{On Some Hadamard-Type Inequalities for
Differentiable $m-$Convex Functions}
\author{$^{\bigstar }$M. Emin \"{O}zdemir}
\address{$^{\bigstar }$Atat\"{u}rk University, K. K. Education Faculty,
Department of Mathematics, 25240, Campus, Erzurum, Turkey }
\email{emos@atauni.edu.tr}
\author{$^{\blacksquare ,\spadesuit }$Ahmet Ocak Akdemir}
\address{$^{\spadesuit }$A\u{g}r\i\ \.{I}brahim \c{C}e\c{c}en University
Faculty of Science and Letters, Department of Mathematics, 04100, A\u{g}r\i
, Turkey}
\email{ahmetakdemir@agri.edu.tr}
\author{$^{\bigstar }$Merve Avc\i }
\keywords{$m-$Convex, Hadamard-Type Inequalities, H\"{o}lder inequality,
Power mean inequality, Favard's inequality.\\
$^{\blacksquare }$Corresponding author}

\begin{abstract}
In this paper some new inequalities are proved related to left hand side of
Hermite-Hadamard inequality for the classes of functions whose derivatives
of absolute values are $m-$convex. New bounds and estimations are obtained.
Applications for some Theorems are given as well.
\end{abstract}

\maketitle

\section{INTRODUCTION}

Let $f:I\rightarrow \mathbb{R}$ be a convex function on the interval $I$ of
real numbers and $a,b\in I$ with $a<b.$ If $f$ is a convex function then the
following double inequality, which is well-known in the literature as
Hermite-Hadamard inequality, holds [see \cite{a}, p. 137]; 
\begin{equation}
f\left( \frac{a+b}{2}\right) \leq \frac{1}{b-a}\int_{a}^{b}f\left( x\right)
dx\leq \frac{f\left( a\right) +f\left( b\right) }{2}.  \label{1.1}
\end{equation}

For recent results, generalizations and new inequalities related to the
inequality presented above see \cite{b}-\cite{6}.

In \cite{TOA}, Toader defined the concept of $m-$convexity as the following;

\begin{definition}
The function $f:\left[ 0,b\right] \rightarrow 
%TCIMACRO{\U{211d} }%
%BeginExpansion
\mathbb{R}
%EndExpansion
,$ $b>0,$ is said to be $m-$convex, where $m\in \left[ 0,1\right] ,$ if for
every $x,y\in \left[ 0,b\right] $ and $t\in \left[ 0,1\right] $ we have 
\begin{equation*}
f\left( tx+m\left( 1-t\right) y\right) \leq tf\left( x\right) +m\left(
1-t\right) f\left( y\right) .
\end{equation*}%
Denote by $K_{m}(b)$ the set of the $m-$convex functions on $\left[ 0,b%
\right] $ for which $f\left( 0\right) \leq 0.$
\end{definition}

Several papers have been written on $m-$convex functions on $\left[ 0,b%
\right] $ and we refer the papers \cite{BOP}, \cite{BPR}, \cite{ST}, \cite%
{TOA}, \cite{MER}, \cite{TOA2}, \cite{SS3}, \cite{2}, \cite{3} and \cite{4}.
In \cite{8}, Dragomir and Agarwal proved following inequality for convex
functions;

\begin{theorem}
Let $f:I\subseteq 
%TCIMACRO{\U{211d} }%
%BeginExpansion
\mathbb{R}
%EndExpansion
\rightarrow 
%TCIMACRO{\U{211d} }%
%BeginExpansion
\mathbb{R}
%EndExpansion
$ , be a differentiable mapping on $I^{0}$ and $a,b\in I$, where $a<b$. If $%
|f^{\prime }|^{q}$ is convex on $[a,b]$ , then the following inequality
holds; 
\begin{equation}
\left\vert \frac{f(a)+f(b)}{2}-\frac{1}{b-a}\dint\limits_{a}^{b}f(x)dx\right%
\vert \leq \frac{\left( b-a\right) \left( \left\vert f^{\prime }\left(
a\right) \right\vert +\left\vert f^{\prime }\left( b\right) \right\vert
\right) }{8}.  \label{l0}
\end{equation}
\end{theorem}

In \cite{6}, Pearce and Pe\v{c}ari\'{c} proved\ the following inequalities
for convex functions;

\begin{theorem}
Let $f:I\subseteq 
%TCIMACRO{\U{211d} }%
%BeginExpansion
\mathbb{R}
%EndExpansion
\rightarrow 
%TCIMACRO{\U{211d} }%
%BeginExpansion
\mathbb{R}
%EndExpansion
$ , be a differentiable mapping on $I^{0}$ and $a,b\in I$, where $a<b$. If $%
|f^{\prime }|^{q}$ is convex on $[a,b]$ for some $q\geq 1$, then 
\begin{equation}
\left\vert \frac{f(a)+f(b)}{2}-\frac{1}{b-a}\dint\limits_{a}^{b}f(x)dx\right%
\vert \leq \frac{b-a}{4}\left( \frac{\left\vert f^{\prime }\left( a\right)
\right\vert ^{q}+\left\vert f^{\prime }\left( b\right) \right\vert ^{q}}{2}%
\right) ^{\frac{1}{q}}  \label{l1}
\end{equation}%
and%
\begin{equation}
\left\vert f\left( \frac{a+b}{2}\right) -\frac{1}{b-a}\dint%
\limits_{a}^{b}f(x)dx\right\vert \leq \frac{b-a}{4}\left( \frac{\left\vert
f^{\prime }\left( a\right) \right\vert ^{q}+\left\vert f^{\prime }\left(
b\right) \right\vert ^{q}}{2}\right) ^{\frac{1}{q}}.  \label{l2}
\end{equation}
\end{theorem}

In \cite{BOP}, Bakula \textit{et al.} proved the following inequality for $%
m- $convex functions;

\begin{theorem}
Let $I$ be an open real interval such that $[0,\infty )\subset I$. Let $%
f:I\rightarrow $ $%
%TCIMACRO{\U{211d} }%
%BeginExpansion
\mathbb{R}
%EndExpansion
$ be a differentiable function on $I$ such that $f^{\prime }\in L[a,b]$,
where $0\leq a<b<\infty $. If $\left\vert f^{\prime }\right\vert ^{q}$ is $%
m- $convex on $\left[ a,b\right] $ for some fixed $m\in \left( 0,1\right] $
and $q\in \lbrack 1,\infty ),$ then;%
\begin{eqnarray*}
&&\left\vert f\left( \frac{a+b}{2}\right) -\frac{1}{b-a}\dint%
\limits_{a}^{b}f(x)dx\right\vert \\
&\leq &\frac{b-a}{4}\min \left\{ \left( \frac{\left\vert f^{\prime }\left(
a\right) \right\vert ^{q}+m\left\vert f^{\prime }\left( \frac{b}{m}\right)
\right\vert ^{q}}{2}\right) ^{\frac{1}{q}},\left( \frac{m\left\vert
f^{\prime }\left( \frac{a}{m}\right) \right\vert ^{q}+\left\vert f^{\prime
}\left( b\right) \right\vert ^{q}}{2}\right) ^{\frac{1}{q}}\right\} .
\end{eqnarray*}
\end{theorem}

In \cite{SS3}, Dragomir established following inequalities of Hadamard-type
similar to above.

\begin{theorem}
Let $f:\left[ 0,\infty \right) \rightarrow 
%TCIMACRO{\U{211d} }%
%BeginExpansion
\mathbb{R}
%EndExpansion
$ be a $m-$convex function with $m\in (0,1].$ If $0\leq a<b<\infty $ and $%
f\in L_{1}\left[ a,b\right] ,$ then one has the inequality:%
\begin{eqnarray}
f\left( \frac{a+b}{2}\right) &\leq &\frac{1}{b-a}\int_{a}^{b}\frac{%
f(x)+mf\left( \frac{x}{m}\right) }{2}dx  \label{m2} \\
&\leq &\frac{m+1}{4}\left[ \frac{f(a)+f\left( b\right) }{2}+m\frac{f(\frac{a%
}{m})+f\left( \frac{b}{m}\right) }{2}\right] .  \notag
\end{eqnarray}
\end{theorem}

The following classical inequality is well-known in the literature as
Favard's inequality (see \cite{F}, \cite[p.216]{pp});

\begin{theorem}
(i) (Favard's inequality) Let $f$ be a non-negative concave function on $%
\left[ a,b\right] $. If $q\geq 1,$ then 
\begin{equation}
\frac{2^{q}}{q+1}\left( \frac{1}{b-a}\int_{a}^{b}f(x)dx\right) ^{q}\geq 
\frac{1}{b-a}\int_{a}^{b}f^{q}(x)dx.  \label{1.9}
\end{equation}%
If $0<q<1$ the reverse inequality holds in (\ref{1.9}).

(ii) (Thunsdorff's inequality) If $f$ is a non-negative, convex function
with $f(a)=0$, then for $q\geq 1$ the reversed inequality holds in (\ref{1.9}%
).
\end{theorem}

Motivated by the above results, in this paper we consider new Hadamard-type
inequalities for functions whose derivatives of absolute values are $m-$%
convex by using fairly elementary analysis and some classical inequalities
like H\"{o}lder inequality, Power-mean inequality and Favard's inequality.
These new results gives new upper bounds for the Theorem 2-3. We also give
some applications.

\section{MAIN\ RESULTS}

To prove our main results, we use following Lemma which was used by Alomari 
\textit{et al. }(see \cite{7}).

\begin{lemma}
Let $f:I\subseteq 
%TCIMACRO{\U{211d} }%
%BeginExpansion
\mathbb{R}
%EndExpansion
\rightarrow 
%TCIMACRO{\U{211d} }%
%BeginExpansion
\mathbb{R}
%EndExpansion
$, be a differentiable mapping on $I$ where $a,b\in I$, with $a<b$. Let $%
f^{\prime }\in L[a,b],$ then the following equality holds;%
\begin{eqnarray*}
&&f\left( \frac{a+b}{2}\right) -\frac{1}{b-a}\dint\limits_{a}^{b}f(x)dx \\
&=&\frac{b-a}{4}\left[ \dint\limits_{0}^{1}tf^{\prime }\left( t\frac{a+b}{2}%
+\left( 1-t\right) a\right) dt+\dint\limits_{0}^{1}\left( t-1\right)
f^{\prime }\left( tb+\left( 1-t\right) \frac{a+b}{2}\right) dt\right] .
\end{eqnarray*}
\end{lemma}

\begin{theorem}
Let $f:\left[ 0,\infty \right) \rightarrow 
%TCIMACRO{\U{211d} }%
%BeginExpansion
\mathbb{R}
%EndExpansion
$, be a differentiable mapping such that $f^{\prime }\in L[a,b].$ If $%
\left\vert f^{\prime }\right\vert $ is $m-$convex on $\left[ a,b\right] ,$
where $0\leq a<b<\infty $ and for some fixed $m\in \left( 0,1\right] ,$ then
the following inequality holds;%
\begin{equation}
\left\vert f\left( \frac{a+b}{2}\right) -\frac{1}{b-a}\dint%
\limits_{a}^{b}f(x)dx\right\vert \leq \min \left\{
T_{1},T_{2},T_{3},T_{4}\right\}  \label{a1}
\end{equation}%
where%
\begin{eqnarray*}
T_{1} &=&\frac{b-a}{12}\left[ 2\left\vert f^{\prime }\left( \frac{a+b}{2}%
\right) \right\vert +m\left[ \frac{\left\vert f^{\prime }\left( \frac{a}{m}%
\right) \right\vert +\left\vert f^{\prime }\left( \frac{b}{m}\right)
\right\vert }{2}\right] \right] , \\
T_{2} &=&\frac{b-a}{12}\left[ \left\vert f^{\prime }\left( \frac{a+b}{2}%
\right) \right\vert +m\left\vert f^{\prime }\left( \frac{a+b}{2m}\right)
\right\vert +\frac{\left\vert f^{\prime }\left( a\right) \right\vert
+m\left\vert f^{\prime }\left( \frac{b}{m}\right) \right\vert }{2}\right] ,
\\
T_{3} &=&\frac{b-a}{12}\left[ \frac{\left\vert f^{\prime }\left( a\right)
\right\vert +\left\vert f^{\prime }\left( b\right) \right\vert }{2}%
+2m\left\vert f^{\prime }\left( \frac{a+b}{2m}\right) \right\vert \right] ,
\\
T_{4} &=&\frac{b-a}{12}\left[ \left\vert f^{\prime }\left( \frac{a+b}{2}%
\right) \right\vert +m\left\vert f^{\prime }\left( \frac{a+b}{2m}\right)
\right\vert +\frac{m\left\vert f^{\prime }\left( \frac{a}{m}\right)
\right\vert +\left\vert f^{\prime }\left( b\right) \right\vert }{2}\right] .
\end{eqnarray*}
\end{theorem}

\begin{proof}
From the equality which is given in the Lemma 1 and by using the properties
of modulus, we have%
\begin{eqnarray}
&&\left\vert f\left( \frac{a+b}{2}\right) -\frac{1}{b-a}\dint%
\limits_{a}^{b}f(x)dx\right\vert  \label{k0} \\
&\leq &\frac{b-a}{4}\left[ \dint\limits_{0}^{1}\left\vert t\right\vert
\left\vert f^{\prime }\left( t\frac{a+b}{2}+\left( 1-t\right) a\right)
\right\vert dt\right.  \notag \\
&&\left. +\dint\limits_{0}^{1}\left\vert t-1\right\vert \left\vert f^{\prime
}\left( tb+\left( 1-t\right) \frac{a+b}{2}\right) \right\vert dt\right] . 
\notag
\end{eqnarray}%
By using $m-$convexity of $\left\vert f^{\prime }\right\vert $ on $\left[ a,b%
\right] ,$ we know that for any $t\in \left[ 0,1\right] $ 
\begin{equation}
\left\vert f^{\prime }\left( t\frac{a+b}{2}+\left( 1-t\right) a\right)
\right\vert \leq t\left\vert f^{\prime }\left( \frac{a+b}{2}\right)
\right\vert +m\left( 1-t\right) \left\vert f^{\prime }\left( \frac{a}{m}%
\right) \right\vert  \label{k1}
\end{equation}%
and%
\begin{equation}
\left\vert f^{\prime }\left( tb+\left( 1-t\right) \frac{a+b}{2}\right)
\right\vert \leq \left( 1-t\right) \left\vert f^{\prime }\left( \frac{a+b}{2}%
\right) \right\vert +mt\left\vert f^{\prime }\left( \frac{b}{m}\right)
\right\vert .  \label{k2}
\end{equation}%
From the inequalities (\ref{k1}) and (\ref{k2}), we obtain%
\begin{eqnarray*}
&&\left\vert f\left( \frac{a+b}{2}\right) -\frac{1}{b-a}\dint%
\limits_{a}^{b}f(x)dx\right\vert \\
&\leq &\frac{b-a}{4}\left[ \dint\limits_{0}^{1}t\left( t\left\vert f^{\prime
}\left( \frac{a+b}{2}\right) \right\vert +m\left( 1-t\right) \left\vert
f^{\prime }\left( \frac{a}{m}\right) \right\vert \right) dt\right. \\
&&\left. +\dint\limits_{0}^{1}\left( 1-t\right) \left( \left( 1-t\right)
\left\vert f^{\prime }\left( \frac{a+b}{2}\right) \right\vert +mt\left\vert
f^{\prime }\left( \frac{b}{m}\right) \right\vert \right) dt\right] .
\end{eqnarray*}%
By calculating the above integrals, we get the following inequality;%
\begin{equation}
\left\vert f\left( \frac{a+b}{2}\right) -\frac{1}{b-a}\dint%
\limits_{a}^{b}f(x)dx\right\vert \leq \frac{b-a}{12}\left[ 2\left\vert
f^{\prime }\left( \frac{a+b}{2}\right) \right\vert +m\left[ \frac{\left\vert
f^{\prime }\left( \frac{a}{m}\right) \right\vert +\left\vert f^{\prime
}\left( \frac{b}{m}\right) \right\vert }{2}\right] \right] .  \label{b1}
\end{equation}%
Analogously, we obtain the following inequalities;%
\begin{equation}
\left\vert f\left( \frac{a+b}{2}\right) -\frac{1}{b-a}\dint%
\limits_{a}^{b}f(x)dx\right\vert \leq \frac{b-a}{12}\left[ \left\vert
f^{\prime }\left( \frac{a+b}{2}\right) \right\vert +m\left\vert f^{\prime
}\left( \frac{a+b}{2m}\right) \right\vert +\frac{\left\vert f^{\prime
}\left( a\right) \right\vert +m\left\vert f^{\prime }\left( \frac{b}{m}%
\right) \right\vert }{2}\right]  \label{b2}
\end{equation}%
\begin{equation}
\left\vert f\left( \frac{a+b}{2}\right) -\frac{1}{b-a}\dint%
\limits_{a}^{b}f(x)dx\right\vert \leq \frac{b-a}{12}\left[ \frac{\left\vert
f^{\prime }\left( a\right) \right\vert +\left\vert f^{\prime }\left(
b\right) \right\vert }{2}+2m\left\vert f^{\prime }\left( \frac{a+b}{2m}%
\right) \right\vert \right]  \label{b3}
\end{equation}%
and 
\begin{equation}
\left\vert f\left( \frac{a+b}{2}\right) -\frac{1}{b-a}\dint%
\limits_{a}^{b}f(x)dx\right\vert \leq \frac{b-a}{12}\left[ \left\vert
f^{\prime }\left( \frac{a+b}{2}\right) \right\vert +m\left\vert f^{\prime
}\left( \frac{a+b}{2m}\right) \right\vert +\frac{m\left\vert f^{\prime
}\left( \frac{a}{m}\right) \right\vert +\left\vert f^{\prime }\left(
b\right) \right\vert }{2}\right] .  \label{b4}
\end{equation}%
From the inequalities (\ref{b1}), (\ref{b2}), (\ref{b3}) and (\ref{b4}), we
get the desired result.
\end{proof}

\begin{corollary}
If we choose $m=1$ in (\ref{a1}), we obtain the inequality; 
\begin{equation*}
\left\vert f\left( \frac{a+b}{2}\right) -\frac{1}{b-a}\dint%
\limits_{a}^{b}f(x)dx\right\vert \leq \frac{b-a}{12}\left[ 2\left\vert
f^{\prime }\left( \frac{a+b}{2}\right) \right\vert +\frac{\left\vert
f^{\prime }\left( a\right) \right\vert +\left\vert f^{\prime }\left(
b\right) \right\vert }{2}\right] .
\end{equation*}
\end{corollary}

\begin{corollary}
Under the assumptions of Theorem 6;

i) If we choose $m=1$ and $\left\vert f^{\prime }\right\vert $ is increasing
in (\ref{a1}), we obtain the inequality;%
\begin{equation*}
\left\vert f\left( \frac{a+b}{2}\right) -\frac{1}{b-a}\dint%
\limits_{a}^{b}f(x)dx\right\vert \leq \frac{b-a}{12}\left[ 2\left\vert
f^{\prime }\left( \frac{a+b}{2}\right) \right\vert +\left\vert f^{\prime
}\left( b\right) \right\vert \right] .
\end{equation*}

ii) If we choose $m=1$ and $\left\vert f^{\prime }\right\vert $ is
decreasing in (\ref{a1}), we obtain the inequality;%
\begin{equation*}
\left\vert f\left( \frac{a+b}{2}\right) -\frac{1}{b-a}\dint%
\limits_{a}^{b}f(x)dx\right\vert \leq \frac{b-a}{12}\left[ 2\left\vert
f^{\prime }\left( \frac{a+b}{2}\right) \right\vert +\left\vert f^{\prime
}\left( a\right) \right\vert \right] .
\end{equation*}

iii) If we choose $m=1$ and $\left\vert f^{\prime }\left( \frac{a+b}{2}%
\right) \right\vert =0$ in (\ref{a1}), we obtain the inequality;%
\begin{equation*}
\left\vert f\left( \frac{a+b}{2}\right) -\frac{1}{b-a}\dint%
\limits_{a}^{b}f(x)dx\right\vert \leq \frac{b-a}{12}\left[ \frac{\left\vert
f^{\prime }\left( a\right) \right\vert +\left\vert f^{\prime }\left(
b\right) \right\vert }{2}\right] .
\end{equation*}

iv) If we choose $m=1$ and $\left\vert f^{\prime }\left( a\right)
\right\vert =\left\vert f^{\prime }\left( b\right) \right\vert =0$ in (\ref%
{a1}), we obtain the inequality;%
\begin{equation*}
\left\vert f\left( \frac{a+b}{2}\right) -\frac{1}{b-a}\dint%
\limits_{a}^{b}f(x)dx\right\vert \leq \frac{b-a}{6}\left\vert f^{\prime
}\left( \frac{a+b}{2}\right) \right\vert .
\end{equation*}
\end{corollary}

\begin{theorem}
Let $f:\left[ 0,\infty \right) \rightarrow 
%TCIMACRO{\U{211d} }%
%BeginExpansion
\mathbb{R}
%EndExpansion
$ , be a differentiable mapping such that $f^{\prime }\in L[a,b].$ If $%
\left\vert f^{\prime }\right\vert ^{\frac{p}{p-1}}$ is $m-$convex on $\left[
a,b\right] ,$ where $0\leq a<b<\infty $, for some fixed $m\in \left( 0,1%
\right] $ and $p>1,$ then the following inequality holds;%
\begin{eqnarray}
\left\vert f\left( \frac{a+b}{2}\right) -\frac{1}{b-a}\dint%
\limits_{a}^{b}f(x)dx\right\vert &\leq &\frac{b-a}{4\left( p+1\right) ^{%
\frac{1}{p}}}\left( \frac{1}{2}\right) ^{\frac{1}{q}}\min \left\{
U_{1},U_{2},U_{3},U_{4}\right\}  \label{a2} \\
&\leq &\frac{b-a}{4}\left( \frac{1}{2}\right) ^{\frac{1}{q}}\min \left\{
U_{1},U_{2},U_{3},U_{4}\right\}  \notag
\end{eqnarray}%
where $\frac{1}{q}+\frac{1}{p}=1$ and 
\begin{eqnarray*}
U_{1} &=&\left( \left\vert f^{\prime }\left( \frac{a+b}{2}\right)
\right\vert ^{q}+m\left\vert f^{\prime }\left( \frac{a}{m}\right)
\right\vert ^{q}\right) ^{\frac{1}{q}}+\left( \left\vert f^{\prime }\left( 
\frac{a+b}{2}\right) \right\vert ^{q}+m\left\vert f^{\prime }\left( \frac{b}{%
m}\right) \right\vert ^{q}\right) ^{\frac{1}{q}}, \\
U_{2} &=&\left( \left\vert f^{\prime }\left( a\right) \right\vert
^{q}+m\left\vert f^{\prime }\left( \frac{a+b}{2m}\right) \right\vert
^{q}\right) ^{\frac{1}{q}}+\left( \left\vert f^{\prime }\left( \frac{a+b}{2}%
\right) \right\vert ^{q}+m\left\vert f^{\prime }\left( \frac{b}{m}\right)
\right\vert ^{q}\right) ^{\frac{1}{q}}, \\
U_{3} &=&\left( \left\vert f^{\prime }\left( a\right) \right\vert
^{q}+m\left\vert f^{\prime }\left( \frac{a+b}{2m}\right) \right\vert
^{q}\right) ^{\frac{1}{q}}+\left( \left\vert f^{\prime }\left( b\right)
\right\vert ^{q}+m\left\vert f^{\prime }\left( \frac{a+b}{2m}\right)
\right\vert ^{q}\right) ^{\frac{1}{q}}, \\
U_{4} &=&\left( \left\vert f^{\prime }\left( \frac{a+b}{2}\right)
\right\vert ^{q}+m\left\vert f^{\prime }\left( \frac{a}{m}\right)
\right\vert ^{q}\right) ^{\frac{1}{q}}+\left( \left\vert f^{\prime }\left( 
\frac{a+b}{2}\right) \right\vert ^{q}+m\left\vert f^{\prime }\left( \frac{b}{%
m}\right) \right\vert ^{q}\right) ^{\frac{1}{q}}.
\end{eqnarray*}
\end{theorem}

\begin{proof}
From Lemma 1 and by using the properties of modulus, we have%
\begin{eqnarray}
&&  \label{z1} \\
&&\left\vert f\left( \frac{a+b}{2}\right) -\frac{1}{b-a}\dint%
\limits_{a}^{b}f(x)dx\right\vert  \notag \\
&\leq &\frac{b-a}{4}\left[ \dint\limits_{0}^{1}\left\vert t\right\vert
\left\vert f^{\prime }\left( t\frac{a+b}{2}+\left( 1-t\right) a\right)
\right\vert dt+\dint\limits_{0}^{1}\left\vert t-1\right\vert \left\vert
f^{\prime }\left( tb+\left( 1-t\right) \frac{a+b}{2}\right) \right\vert dt%
\right] .  \notag
\end{eqnarray}%
By applying the H\"{o}lder inequality to the inequality (\ref{z1}), we get%
\begin{eqnarray*}
&&\left\vert f\left( \frac{a+b}{2}\right) -\frac{1}{b-a}\dint%
\limits_{a}^{b}f(x)dx\right\vert \\
&\leq &\frac{b-a}{4}\left[ \left( \dint\limits_{0}^{1}t^{p}dt\right) ^{\frac{%
1}{p}}\left( \dint\limits_{0}^{1}\left\vert f^{\prime }\left( t\frac{a+b}{2}%
+\left( 1-t\right) a\right) \right\vert ^{q}dt\right) ^{\frac{1}{q}}\right.
\\
&&\left. +\left( \dint\limits_{0}^{1}\left( 1-t\right) ^{p}dt\right) ^{\frac{%
1}{p}}\left( \dint\limits_{0}^{1}\left\vert f^{\prime }\left( tb+\left(
1-t\right) \frac{a+b}{2}\right) \right\vert ^{q}dt\right) ^{\frac{1}{q}}%
\right] .
\end{eqnarray*}%
It is easy to see that%
\begin{equation*}
\dint\limits_{0}^{1}t^{p}dt=\dint\limits_{0}^{1}\left( 1-t\right) ^{p}dt=%
\frac{1}{p+1}.
\end{equation*}%
Hence, by $m-$convexity of $\left\vert f^{\prime }\right\vert ^{q}$ on $%
\left[ a,b\right] ,$ we obtain the inequality;%
\begin{eqnarray}
&&\left\vert f\left( \frac{a+b}{2}\right) -\frac{1}{b-a}\dint%
\limits_{a}^{b}f(x)dx\right\vert  \label{z2} \\
&\leq &\frac{b-a}{4\left( p+1\right) ^{\frac{1}{p}}}\left( \frac{1}{2}%
\right) ^{\frac{1}{q}}\left[ \left( \left\vert f^{\prime }\left( \frac{a+b}{2%
}\right) \right\vert ^{q}+m\left\vert f^{\prime }\left( \frac{a}{m}\right)
\right\vert ^{q}\right) ^{\frac{1}{q}}\right.  \notag \\
&&\left. +\left( \left\vert f^{\prime }\left( \frac{a+b}{2}\right)
\right\vert ^{q}+m\left\vert f^{\prime }\left( \frac{b}{m}\right)
\right\vert ^{q}\right) ^{\frac{1}{q}}\right] .  \notag
\end{eqnarray}%
By a similar argument to the proof of Theorem 6, analogously, we obtain the
following inequalities; 
\begin{eqnarray}
&&\left\vert f\left( \frac{a+b}{2}\right) -\frac{1}{b-a}\dint%
\limits_{a}^{b}f(x)dx\right\vert  \label{z3} \\
&\leq &\frac{b-a}{4\left( p+1\right) ^{\frac{1}{p}}}\left( \frac{1}{2}%
\right) ^{\frac{1}{q}}\left[ \left( \left\vert f^{\prime }\left( a\right)
\right\vert ^{q}+m\left\vert f^{\prime }\left( \frac{a+b}{2m}\right)
\right\vert ^{q}\right) ^{\frac{1}{q}}\right.  \notag \\
&&\left. +\left( \left\vert f^{\prime }\left( \frac{a+b}{2}\right)
\right\vert ^{q}+m\left\vert f^{\prime }\left( \frac{b}{m}\right)
\right\vert ^{q}\right) ^{\frac{1}{q}}\right] ,  \notag
\end{eqnarray}%
\begin{eqnarray}
&&\left\vert f\left( \frac{a+b}{2}\right) -\frac{1}{b-a}\dint%
\limits_{a}^{b}f(x)dx\right\vert  \label{z4} \\
&\leq &\frac{b-a}{4\left( p+1\right) ^{\frac{1}{p}}}\left( \frac{1}{2}%
\right) ^{\frac{1}{q}}\left[ \left( \left\vert f^{\prime }\left( a\right)
\right\vert ^{q}+m\left\vert f^{\prime }\left( \frac{a+b}{2m}\right)
\right\vert ^{q}\right) ^{\frac{1}{q}}\right.  \notag \\
&&\left. +\left( \left\vert f^{\prime }\left( b\right) \right\vert
^{q}+m\left\vert f^{\prime }\left( \frac{a+b}{2m}\right) \right\vert
^{q}\right) ^{\frac{1}{q}}\right]  \notag
\end{eqnarray}%
and%
\begin{eqnarray}
&&\left\vert f\left( \frac{a+b}{2}\right) -\frac{1}{b-a}\dint%
\limits_{a}^{b}f(x)dx\right\vert  \label{z5} \\
&\leq &\frac{b-a}{4\left( p+1\right) ^{\frac{1}{p}}}\left( \frac{1}{2}%
\right) ^{\frac{1}{q}}\left[ \left( \left\vert f^{\prime }\left( \frac{a+b}{2%
}\right) \right\vert ^{q}+m\left\vert f^{\prime }\left( \frac{a}{m}\right)
\right\vert ^{q}\right) ^{\frac{1}{q}}\right.  \notag \\
&&\left. +\left( \left\vert f^{\prime }\left( \frac{a+b}{2}\right)
\right\vert ^{q}+m\left\vert f^{\prime }\left( \frac{b}{m}\right)
\right\vert ^{q}\right) ^{\frac{1}{q}}\right] .  \notag
\end{eqnarray}%
From the inequalities (\ref{z2})-(\ref{z5}), we obtain the inequality in (%
\ref{a2}). The second inequality in (\ref{a2}) follows from facts that; 
\begin{equation*}
\lim_{p\rightarrow \infty }\left( \frac{1}{1+p}\right) ^{\frac{1}{p}}=1\text{
\ \ \ \ , \ \ \ }\lim_{p\rightarrow 1^{+}}\left( \frac{1}{1+p}\right) ^{%
\frac{1}{p}}=\frac{1}{2}
\end{equation*}%
and%
\begin{equation*}
\frac{1}{2}<\left( \frac{1}{1+p}\right) ^{\frac{1}{p}}<1.
\end{equation*}
\end{proof}

\begin{corollary}
Under the assumptions of Theorem 7, if we choose $m=1,$ we obtain the
inequality;%
\begin{eqnarray*}
&&\left\vert f\left( \frac{a+b}{2}\right) -\frac{1}{b-a}\dint%
\limits_{a}^{b}f(x)dx\right\vert \\
&\leq &\frac{b-a}{4\left( p+1\right) ^{\frac{1}{p}}}\left( \frac{1}{2}%
\right) ^{\frac{1}{q}}\left[ \left( \left\vert f^{\prime }\left( \frac{a+b}{2%
}\right) \right\vert ^{q}+\left\vert f^{\prime }\left( a\right) \right\vert
^{q}\right) ^{\frac{1}{q}}\right. \\
&&\left. +\left( \left\vert f^{\prime }\left( \frac{a+b}{2}\right)
\right\vert ^{q}+\left\vert f^{\prime }\left( b\right) \right\vert
^{q}\right) ^{\frac{1}{q}}\right] .
\end{eqnarray*}
\end{corollary}

\begin{corollary}
Under the assumptions of Theorem 7;

i) If we choose $m=1$ and $\left\vert f^{\prime }\right\vert ^{\frac{p}{p-1}%
} $ is increasing in (\ref{a2}), we obtain the inequality;%
\begin{equation*}
\left\vert f\left( \frac{a+b}{2}\right) -\frac{1}{b-a}\dint%
\limits_{a}^{b}f(x)dx\right\vert \leq \frac{b-a}{4\left( p+1\right) ^{\frac{1%
}{p}}}\left( \left\vert f^{\prime }\left( \frac{a+b}{2}\right) \right\vert
^{q}+\left\vert f^{\prime }\left( b\right) \right\vert ^{q}\right) ^{\frac{1%
}{q}}.
\end{equation*}

ii) If we choose $m=1$ and $\left\vert f^{\prime }\right\vert ^{\frac{p}{p-1}%
}$ is decreasing in (\ref{a2}), we obtain the inequality;%
\begin{equation*}
\left\vert f\left( \frac{a+b}{2}\right) -\frac{1}{b-a}\dint%
\limits_{a}^{b}f(x)dx\right\vert \leq \frac{b-a}{4\left( p+1\right) ^{\frac{1%
}{p}}}\left( \left\vert f^{\prime }\left( \frac{a+b}{2}\right) \right\vert
^{q}+\left\vert f^{\prime }\left( a\right) \right\vert ^{q}\right) ^{\frac{1%
}{q}}.
\end{equation*}

iii) If we choose $m=1$ and $\left\vert f^{\prime }\left( \frac{a+b}{2}%
\right) \right\vert ^{\frac{p}{p-1}}=0$ in (\ref{a2}), we obtain the
inequality;%
\begin{equation*}
\left\vert f\left( \frac{a+b}{2}\right) -\frac{1}{b-a}\dint%
\limits_{a}^{b}f(x)dx\right\vert \leq \frac{b-a}{4\left( p+1\right) ^{\frac{1%
}{p}}}\left( \frac{1}{2}\right) ^{\frac{1}{q}}\left( \left\vert f^{\prime
}\left( a\right) \right\vert +\left\vert f^{\prime }\left( b\right)
\right\vert \right) .
\end{equation*}

iv) If we choose $m=1$ and $\left\vert f^{\prime }\left( a\right)
\right\vert ^{\frac{p}{p-1}}=\left\vert f^{\prime }\left( b\right)
\right\vert ^{\frac{p}{p-1}}=0$ in (\ref{a2}), we obtain the inequality;%
\begin{equation*}
\left\vert f\left( \frac{a+b}{2}\right) -\frac{1}{b-a}\dint%
\limits_{a}^{b}f(x)dx\right\vert \leq \frac{b-a}{4\left( p+1\right) ^{\frac{1%
}{p}}}\left( \frac{1}{2}\right) ^{\frac{1}{q}}\left\vert f^{\prime }\left( 
\frac{a+b}{2}\right) \right\vert .
\end{equation*}
\end{corollary}

\begin{theorem}
Let $f:\left[ 0,\infty \right) \rightarrow 
%TCIMACRO{\U{211d} }%
%BeginExpansion
\mathbb{R}
%EndExpansion
$ , be a differentiable mapping such that $f^{\prime }\in L[a,b].$ If $%
\left\vert f^{\prime }\right\vert ^{q}$ is $m-$convex on $\left[ a,b\right]
, $ where $0\leq a<b<\infty $, for some fixed $m\in \left( 0,1\right] $ and $%
q\geq 1,$ then the following inequality holds;%
\begin{equation}
\left\vert f\left( \frac{a+b}{2}\right) -\frac{1}{b-a}\dint%
\limits_{a}^{b}f(x)dx\right\vert \leq \frac{b-a}{4}\left( \frac{1}{2}\right)
^{1-\frac{1}{q}}\min \left\{ V_{1},V_{2},V_{3},V_{4}\right\}  \label{a3}
\end{equation}%
where%
\begin{eqnarray*}
V_{1} &=&\left( \frac{1}{3}\left\vert f^{\prime }\left( \frac{a+b}{2}\right)
\right\vert ^{q}+\frac{m}{6}\left\vert f^{\prime }\left( \frac{a}{m}\right)
\right\vert ^{q}\right) ^{\frac{1}{q}}+\left( \frac{1}{3}\left\vert
f^{\prime }\left( \frac{a+b}{2}\right) \right\vert ^{q}+\frac{m}{6}%
\left\vert f^{\prime }\left( \frac{b}{m}\right) \right\vert ^{q}\right) ^{%
\frac{1}{q}}, \\
V_{2} &=&\left( \frac{1}{6}\left\vert f^{\prime }\left( a\right) \right\vert
^{q}+\frac{m}{3}\left\vert f^{\prime }\left( \frac{a+b}{2m}\right)
\right\vert ^{q}\right) ^{\frac{1}{q}}+\left( \frac{1}{3}\left\vert
f^{\prime }\left( \frac{a+b}{2}\right) \right\vert ^{q}+\frac{m}{6}%
\left\vert f^{\prime }\left( \frac{b}{m}\right) \right\vert ^{q}\right) ^{%
\frac{1}{q}}, \\
V_{3} &=&\left( \frac{1}{6}\left\vert f^{\prime }\left( a\right) \right\vert
^{q}+\frac{m}{3}\left\vert f^{\prime }\left( \frac{a+b}{2m}\right)
\right\vert ^{q}\right) ^{\frac{1}{q}}+\left( \frac{1}{6}\left\vert
f^{\prime }\left( b\right) \right\vert ^{q}+\frac{m}{3}\left\vert f^{\prime
}\left( \frac{a+b}{2m}\right) \right\vert ^{q}\right) ^{\frac{1}{q}}, \\
V_{4} &=&\left( \frac{1}{3}\left\vert f^{\prime }\left( \frac{a+b}{2}\right)
\right\vert ^{q}+\frac{m}{6}\left\vert f^{\prime }\left( \frac{a}{m}\right)
\right\vert ^{q}\right) ^{\frac{1}{q}}+\left( \frac{1}{6}\left\vert
f^{\prime }\left( b\right) \right\vert ^{q}+\frac{m}{3}\left\vert f^{\prime
}\left( \frac{a+b}{2m}\right) \right\vert ^{q}\right) ^{\frac{1}{q}}.
\end{eqnarray*}
\end{theorem}

\begin{proof}
From Lemma 1, we can write%
\begin{eqnarray*}
&&\left\vert f\left( \frac{a+b}{2}\right) -\frac{1}{b-a}\dint%
\limits_{a}^{b}f(x)dx\right\vert \\
&\leq &\frac{b-a}{4}\left[ \dint\limits_{0}^{1}\left\vert t\right\vert
\left\vert f^{\prime }\left( t\frac{a+b}{2}+\left( 1-t\right) a\right)
\right\vert dt+\dint\limits_{0}^{1}\left\vert t-1\right\vert \left\vert
f^{\prime }\left( tb+\left( 1-t\right) \frac{a+b}{2}\right) \right\vert dt%
\right] .
\end{eqnarray*}%
By applying the Power-mean inequality$,$ we get%
\begin{eqnarray*}
&&\left\vert f\left( \frac{a+b}{2}\right) -\frac{1}{b-a}\dint%
\limits_{a}^{b}f(x)dx\right\vert \\
&\leq &\frac{b-a}{4}\left[ \left( \dint\limits_{0}^{1}tdt\right) ^{1-\frac{1%
}{q}}\left( \dint\limits_{0}^{1}t\left\vert f^{\prime }\left( t\frac{a+b}{2}%
+\left( 1-t\right) a\right) \right\vert ^{q}dt\right) ^{\frac{1}{q}}\right.
\\
&&\left. +\left( \dint\limits_{0}^{1}(1-t)dt\right) ^{1-\frac{1}{q}}\left(
\dint\limits_{0}^{1}(1-t)\left\vert f^{\prime }\left( tb+\left( 1-t\right) 
\frac{a+b}{2}\right) \right\vert ^{q}dt\right) ^{\frac{1}{q}}\right] .
\end{eqnarray*}%
Now by using $m-$convexity of $\left\vert f^{\prime }\right\vert ^{q}$ on $%
\left[ a,b\right] $ and by computing the integrals, we obtain the following
inequality;%
\begin{eqnarray}
&&\left\vert f\left( \frac{a+b}{2}\right) -\frac{1}{b-a}\dint%
\limits_{a}^{b}f(x)dx\right\vert  \label{v1} \\
&\leq &\frac{b-a}{4}\left( \frac{1}{2}\right) ^{1-\frac{1}{q}}\left[ \left( 
\frac{1}{3}\left\vert f^{\prime }\left( \frac{a+b}{2}\right) \right\vert
^{q}+\frac{m}{6}\left\vert f^{\prime }\left( \frac{a}{m}\right) \right\vert
^{q}\right) ^{\frac{1}{q}}\right.  \notag \\
&&\left. +\left( \frac{1}{3}\left\vert f^{\prime }\left( \frac{a+b}{2}%
\right) \right\vert ^{q}+\frac{m}{6}\left\vert f^{\prime }\left( \frac{b}{m}%
\right) \right\vert ^{q}\right) ^{\frac{1}{q}}\right] .  \notag
\end{eqnarray}%
Hence, by a similar argument to the proofs of Theorem 6-7, analogously, we
obtain the following inequalities;%
\begin{eqnarray}
&&\left\vert f\left( \frac{a+b}{2}\right) -\frac{1}{b-a}\dint%
\limits_{a}^{b}f(x)dx\right\vert  \label{v2} \\
&\leq &\frac{b-a}{4}\left( \frac{1}{2}\right) ^{1-\frac{1}{q}}\left[ \left( 
\frac{1}{6}\left\vert f^{\prime }\left( a\right) \right\vert ^{q}+\frac{m}{3}%
\left\vert f^{\prime }\left( \frac{a+b}{2m}\right) \right\vert ^{q}\right) ^{%
\frac{1}{q}}\right.  \notag \\
&&\left. +\left( \frac{1}{3}\left\vert f^{\prime }\left( \frac{a+b}{2}%
\right) \right\vert ^{q}+\frac{m}{6}\left\vert f^{\prime }\left( \frac{b}{m}%
\right) \right\vert ^{q}\right) ^{\frac{1}{q}}\right] ,  \notag
\end{eqnarray}%
\begin{eqnarray}
&&\left\vert f\left( \frac{a+b}{2}\right) -\frac{1}{b-a}\dint%
\limits_{a}^{b}f(x)dx\right\vert  \label{v3} \\
&\leq &\frac{b-a}{4}\left( \frac{1}{2}\right) ^{1-\frac{1}{q}}\left[ \left( 
\frac{1}{6}\left\vert f^{\prime }\left( a\right) \right\vert ^{q}+\frac{m}{3}%
\left\vert f^{\prime }\left( \frac{a+b}{2m}\right) \right\vert ^{q}\right) ^{%
\frac{1}{q}}\right.  \notag \\
&&\left. +\left( \frac{1}{6}\left\vert f^{\prime }\left( b\right)
\right\vert ^{q}+\frac{m}{3}\left\vert f^{\prime }\left( \frac{a+b}{2m}%
\right) \right\vert ^{q}\right) ^{\frac{1}{q}}\right] ,  \notag
\end{eqnarray}%
and%
\begin{eqnarray}
&&\left\vert f\left( \frac{a+b}{2}\right) -\frac{1}{b-a}\dint%
\limits_{a}^{b}f(x)dx\right\vert  \label{v4} \\
&\leq &\frac{b-a}{4}\left( \frac{1}{2}\right) ^{1-\frac{1}{q}}\left[ \left( 
\frac{1}{3}\left\vert f^{\prime }\left( \frac{a+b}{2}\right) \right\vert
^{q}+\frac{m}{6}\left\vert f^{\prime }\left( \frac{a}{m}\right) \right\vert
^{q}\right) ^{\frac{1}{q}}\right.  \notag \\
&&\left. +\left( \frac{1}{6}\left\vert f^{\prime }\left( b\right)
\right\vert ^{q}+\frac{m}{3}\left\vert f^{\prime }\left( \frac{a+b}{2m}%
\right) \right\vert ^{q}\right) ^{\frac{1}{q}}\right] .  \notag
\end{eqnarray}%
By the inequalities (\ref{v1})-(\ref{v4}), we obtain the inequality (\ref{a3}%
).
\end{proof}

\begin{corollary}
Under the assumptions of Theorem 8, if we choose $m=1,$ we obtain the
inequality;%
\begin{eqnarray*}
&&\left\vert f\left( \frac{a+b}{2}\right) -\frac{1}{b-a}\dint%
\limits_{a}^{b}f(x)dx\right\vert \\
&\leq &\frac{b-a}{4}\left( \frac{1}{2}\right) ^{1-\frac{1}{q}}\left[ \left( 
\frac{1}{3}\left\vert f^{\prime }\left( \frac{a+b}{2}\right) \right\vert
^{q}+\frac{1}{6}\left\vert f^{\prime }\left( a\right) \right\vert
^{q}\right) ^{\frac{1}{q}}\right. \\
&&\left. +\left( \frac{1}{6}\left\vert f^{\prime }\left( b\right)
\right\vert ^{q}+\frac{1}{3}\left\vert f^{\prime }\left( \frac{a+b}{2}%
\right) \right\vert ^{q}\right) ^{\frac{1}{q}}\right] .
\end{eqnarray*}
\end{corollary}

\begin{corollary}
Under the assumptions of Theorem 8;

i) If we choose $m=1$ and $\left\vert f^{\prime }\right\vert ^{q}$ is
increasing in (\ref{a3}), we obtain the inequality;%
\begin{equation*}
\left\vert f\left( \frac{a+b}{2}\right) -\frac{1}{b-a}\dint%
\limits_{a}^{b}f(x)dx\right\vert \leq \frac{b-a}{4}\left( \frac{1}{2}\right)
^{1-\frac{1}{q}}\left( \frac{1}{3}\left\vert f^{\prime }\left( \frac{a+b}{2}%
\right) \right\vert ^{q}+\frac{1}{6}\left\vert f^{\prime }\left( b\right)
\right\vert ^{q}\right) ^{\frac{1}{q}}.
\end{equation*}

ii) If we choose $m=1$ and $\left\vert f^{\prime }\right\vert ^{q}$ is
decreasing in (\ref{a3}), we obtain the inequality;%
\begin{equation*}
\left\vert f\left( \frac{a+b}{2}\right) -\frac{1}{b-a}\dint%
\limits_{a}^{b}f(x)dx\right\vert \leq \frac{b-a}{4}\left( \frac{1}{2}\right)
^{1-\frac{1}{q}}\left( \frac{1}{3}\left\vert f^{\prime }\left( \frac{a+b}{2}%
\right) \right\vert ^{q}+\frac{1}{6}\left\vert f^{\prime }\left( a\right)
\right\vert ^{q}\right) ^{\frac{1}{q}}.
\end{equation*}

iii) If we choose $m=1$ and $\left\vert f^{\prime }\left( \frac{a+b}{2}%
\right) \right\vert ^{q}=0$ in (\ref{a3}), we obtain the inequality;%
\begin{equation*}
\left\vert f\left( \frac{a+b}{2}\right) -\frac{1}{b-a}\dint%
\limits_{a}^{b}f(x)dx\right\vert \leq \frac{b-a}{8}\left( \frac{1}{3}\right)
^{\frac{1}{q}}\left( \left\vert f^{\prime }\left( a\right) \right\vert
+\left\vert f^{\prime }\left( b\right) \right\vert \right) .
\end{equation*}

iv) If we choose $m=1$ and $\left\vert f^{\prime }\left( a\right)
\right\vert ^{q}=\left\vert f^{\prime }\left( b\right) \right\vert ^{q}=0$
in (\ref{a3}), we obtain the inequality;%
\begin{equation*}
\left\vert f\left( \frac{a+b}{2}\right) -\frac{1}{b-a}\dint%
\limits_{a}^{b}f(x)dx\right\vert \leq \frac{b-a}{4}\left( \frac{1}{6}\right)
^{1-\frac{1}{q}}\left\vert f^{\prime }\left( \frac{a+b}{2}\right)
\right\vert .
\end{equation*}
\end{corollary}

\begin{theorem}
Let $f,g:[0,b]\rightarrow \mathbb{R}$, be concave and $m-$concave functions, 
$m\in (0,1],$ where $0\leq a<b<\infty $ and $q\geq 1.$ Then 
\begin{eqnarray*}
f\left( \frac{a+b}{2}\right) g\left( \frac{a+b}{2}\right) &\geq &\frac{%
\left( p+1\right) ^{\frac{1}{p}}\left( q+1\right) ^{\frac{1}{q}}}{16} \\
&&\times \left( \frac{1}{b-a}\dint\limits_{a}^{b}\left[ f(x)+mf\left( \frac{x%
}{m}\right) \right] \left[ g(x)+mg\left( \frac{x}{m}\right) \right]
dx\right) .
\end{eqnarray*}%
where $\frac{1}{q}+\frac{1}{p}=1.$

If $f,g$ are convex and $m-$convex functions, with $f(0)=0,$ then the
reverse of the above inequality holds.
\end{theorem}

\begin{proof}
Since $f,g$ are $m-$concave, by using the inequality (\ref{m2}), we can write%
\begin{equation*}
f\left( \frac{a+b}{2}\right) \geq \frac{1}{b-a}\int_{a}^{b}\frac{%
f(x)+mf\left( \frac{x}{m}\right) }{2}dx
\end{equation*}%
and%
\begin{equation*}
g\left( \frac{a+b}{2}\right) \leq \frac{1}{b-a}\int_{a}^{b}\frac{%
g(x)+mg\left( \frac{x}{m}\right) }{2}dx.
\end{equation*}%
By using Favard's inequality for $p-$th powers of both sides of inequality,
we have 
\begin{eqnarray*}
f^{p}\left( \frac{a+b}{2}\right) &\geq &\left( \frac{1}{b-a}%
\dint\limits_{a}^{b}\frac{f(x)+mf\left( \frac{x}{m}\right) }{2}dx\right) ^{p}
\\
&\geq &\frac{p+1}{2^{p}}\left[ \frac{1}{b-a}\dint\limits_{a}^{b}\left( \frac{%
f(x)+mf\left( \frac{x}{m}\right) }{2}\right) ^{p}dx\right]
\end{eqnarray*}%
and similarly, we have%
\begin{equation*}
g^{q}\left( \frac{a+b}{2}\right) \geq \frac{q+1}{2^{q}}\left[ \frac{1}{b-a}%
\dint\limits_{a}^{b}\left( \frac{g(x)+mg\left( \frac{x}{m}\right) }{2}%
\right) ^{q}dx\right] .
\end{equation*}%
It follows that 
\begin{equation*}
f\left( \frac{a+b}{2}\right) \geq \frac{\left( p+1\right) ^{\frac{1}{p}}}{2}%
\left[ \frac{1}{b-a}\dint\limits_{a}^{b}\left( \frac{f(x)+mf\left( \frac{x}{m%
}\right) }{2}\right) ^{p}dx\right] ^{\frac{1}{p}}
\end{equation*}%
and%
\begin{equation*}
g\left( \frac{a+b}{2}\right) \geq \frac{\left( q+1\right) ^{\frac{1}{q}}}{2}%
\left[ \frac{1}{b-a}\dint\limits_{a}^{b}\left( \frac{g(x)+mg\left( \frac{x}{m%
}\right) }{2}\right) ^{q}dx\right] ^{\frac{1}{q}}.
\end{equation*}%
By multiplying both sides of the above inequalities, we get%
\begin{eqnarray*}
f\left( \frac{a+b}{2}\right) g\left( \frac{a+b}{2}\right) &\geq &\frac{%
\left( p+1\right) ^{\frac{1}{p}}\left( q+1\right) ^{\frac{1}{q}}}{4}\left( 
\frac{1}{b-a}\dint\limits_{a}^{b}\left( \frac{f(x)+mf\left( \frac{x}{m}%
\right) }{2}\right) ^{p}dx\right) ^{\frac{1}{p}} \\
&&\times \left( \frac{1}{b-a}\dint\limits_{a}^{b}\left( \frac{g(x)+mg\left( 
\frac{x}{m}\right) }{2}\right) ^{q}dx\right) ^{\frac{1}{q}}.
\end{eqnarray*}%
By using H\"{o}lder inequality, we have%
\begin{eqnarray*}
f\left( \frac{a+b}{2}\right) g\left( \frac{a+b}{2}\right) &\geq &\frac{%
\left( p+1\right) ^{\frac{1}{p}}\left( q+1\right) ^{\frac{1}{q}}}{16} \\
&&\times \left( \frac{1}{b-a}\dint\limits_{a}^{b}\left[ f(x)+mf\left( \frac{x%
}{m}\right) \right] \left[ g(x)+mg\left( \frac{x}{m}\right) \right]
dx\right) .
\end{eqnarray*}

If $f,g$ are $m-$convex, then using Thunsdorff's inequality we obtain
desired result.
\end{proof}

\begin{corollary}
Under the assumptions of Theorem 9, if we choose $m=1,$ we obtain the
inequality;%
\begin{eqnarray*}
f\left( \frac{a+b}{2}\right) g\left( \frac{a+b}{2}\right) &\geq &\frac{%
\left( p+1\right) ^{\frac{1}{p}}\left( q+1\right) ^{\frac{1}{q}}}{4} \\
&&\times \left( \frac{1}{b-a}\dint\limits_{a}^{b}f(x)g(x)dx\right) .
\end{eqnarray*}
\end{corollary}

\section{\protect\LARGE APPLICATIONS TO\ SOME\ SPECIAL\ MEANS}

We now consider the applications of our Theorems to the following special
means

a) The arithmetic mean:%
\begin{equation*}
A=A\left( a,b\right) :=\frac{a+b}{2},\text{\ \ }a,b\geq 0,
\end{equation*}

b) The logarithmic mean:

\begin{equation*}
L=L\left( a,b\right) :=\left\{ 
\begin{array}{l}
a\text{ \ \ \ \ \ \ \ \ \ \ \ \ \ if \ \ }a=b \\ 
\frac{b-a}{\ln b-\ln a}\text{ \ \ \ \ \ if \ \ }a\neq b%
\end{array}%
\right. ,\text{ \ }a,b\geq 0,
\end{equation*}

c) The $p-$logarithmic mean:

\begin{equation*}
L_{p}=L_{p}\left( a,b\right) :=\left\{ 
\begin{array}{l}
\left[ \frac{b^{p+1}-a^{p+1}}{\left( p+1\right) \left( b-a\right) }\right]
^{1/p}\text{ \ \ \ \ \ if \ \ }a\neq b \\ 
a\text{ \ \ \ \ \ \ \ \ \ \ \ \ \ \ \ \ \ \ \ \ \ \ \ \ if \ \ }a=b%
\end{array}%
\right. ,\text{ \ }p\in 
%TCIMACRO{\U{211d} }%
%BeginExpansion
\mathbb{R}
%EndExpansion
\backslash \left\{ -1,0\right\} ;\text{ \ }a,b>0.
\end{equation*}

We now derive some sophisticated bounds of the above means.

\begin{proposition}
Let $a,b\in 
%TCIMACRO{\U{211d} }%
%BeginExpansion
\mathbb{R}
%EndExpansion
,$ $0<a<b$ and $n\in 
%TCIMACRO{\U{2124} }%
%BeginExpansion
\mathbb{Z}
%EndExpansion
,$ $\left\vert n\right\vert \geq 2.$ Then, we have:%
\begin{equation*}
\left\vert A^{n}\left( a,b\right) -L_{n}^{n}\left( a,b\right) \right\vert
\leq \min \left\{ K_{1},K_{2},K_{3},K_{4}\right\}
\end{equation*}%
where%
\begin{eqnarray*}
K_{1} &=&\frac{n\left( b-a\right) }{12}\left[ 2\left\vert A\left( a,b\right)
\right\vert ^{n-1}+m\left[ A\left( \left\vert \left( \frac{a}{m}\right)
\right\vert ^{n-1},\left\vert \left( \frac{b}{m}\right) \right\vert
^{n-1}\right) \right] \right] , \\
K_{2} &=&\frac{n\left( b-a\right) }{12}\left[ \left\vert A\left( a,b\right)
\right\vert ^{n-1}+m\left\vert \frac{A\left( a,b\right) }{m}\right\vert
^{n-1}+A\left( \left\vert a\right\vert ^{n-1},m\left\vert \frac{b}{m}%
\right\vert ^{n-1}\right) \right] , \\
K_{3} &=&\frac{n\left( b-a\right) }{12}\left[ A\left( \left\vert
a\right\vert ^{n-1}+\left\vert b\right\vert ^{n-1}\right) +2m\left\vert 
\frac{A\left( a,b\right) }{m}\right\vert ^{n-1}\right] , \\
K_{4} &=&\frac{n\left( b-a\right) }{12}\left[ \left\vert A\left( a,b\right)
\right\vert ^{n-1}+m\left\vert \frac{A\left( a,b\right) }{m}\right\vert
^{n-1}+A\left( m\left\vert \frac{a}{m}\right\vert ^{n-1},\left\vert
b\right\vert ^{n-1}\right) \right] .
\end{eqnarray*}
\end{proposition}

\begin{proof}
The proof is immediate from Theorem 6 applied for $f(x)=x^{n},$ which\ is an 
$m-$convex function.
\end{proof}

\begin{proposition}
Let $a,b\in 
%TCIMACRO{\U{211d} }%
%BeginExpansion
\mathbb{R}
%EndExpansion
,$ $0<a<b$ and $n\in 
%TCIMACRO{\U{2124} }%
%BeginExpansion
\mathbb{Z}
%EndExpansion
,\left\vert n\right\vert \geq 2,$ $k\geq 1.$ Then, we have:%
\begin{equation*}
\left\vert A^{\frac{n}{k}}\left( a,b\right) -L_{n}^{\frac{n}{k}}\left(
a,b\right) \right\vert \leq \frac{b-a}{4}\left( \frac{1}{2}\right) ^{1-\frac{%
1}{q}}\min \left\{ L_{1},L_{2},L_{3},L_{4}\right\}
\end{equation*}%
where%
\begin{eqnarray*}
L_{1} &=&\frac{n}{k}2A\left[ \left( 2A\left( \frac{1}{3}\left\vert \left(
A(a,b)\right) \right\vert ^{\frac{q(n-k)}{k}},\frac{m}{6}\left\vert \frac{a}{%
m}\right\vert ^{\frac{q(n-k)}{k}}\right) \right) ^{\frac{1}{q}}\right. \\
&&\left. +\left( 2A\left( \frac{1}{3}\left\vert \left( A(a,b)\right)
\right\vert ^{\frac{q(n-k)}{k}},\frac{m}{6}\left\vert \frac{b}{m}\right\vert
^{\frac{q(n-k)}{k}}\right) \right) ^{\frac{1}{q}}\right] ,
\end{eqnarray*}%
\begin{eqnarray*}
L_{2} &=&\frac{n}{k}2A\left[ \left( 2A\left( \frac{1}{6}\left\vert
a\right\vert ^{\frac{q(n-k)}{k}}+\frac{m}{3}\left\vert \frac{A(a,b)}{m}%
\right\vert ^{\frac{q(n-k)}{k}}\right) \right) ^{\frac{1}{q}}\right. \\
&&\left. +\left( 2A\left( \frac{1}{3}\left\vert \left( A(a,b)\right)
\right\vert ^{\frac{q(n-k)}{k}},\frac{m}{6}\left\vert \frac{b}{m}\right\vert
^{\frac{q(n-k)}{k}}\right) \right) ^{\frac{1}{q}}\right] ,
\end{eqnarray*}%
\begin{eqnarray*}
L_{3} &=&\frac{n}{k}2A\left[ \left( 2A\left( \frac{1}{6}\left\vert
a\right\vert ^{\frac{q(n-k)}{k}}+\frac{m}{3}\left\vert \frac{A(a,b)}{m}%
\right\vert ^{\frac{q(n-k)}{k}}\right) \right) ^{\frac{1}{q}}\right. \\
&&\left. +\left( 2A\left( \frac{1}{6}\left\vert b\right\vert ^{\frac{q(n-k)}{%
k}}+\frac{m}{3}\left\vert \frac{A(a,b)}{m}\right\vert ^{\frac{q(n-k)}{k}%
}\right) \right) ^{\frac{1}{q}}\right] ,
\end{eqnarray*}%
\begin{eqnarray*}
L_{4} &=&\frac{n}{k}2A\left[ \left( 2A\left( \frac{1}{3}\left\vert \left(
A(a,b)\right) \right\vert ^{\frac{q(n-k)}{k}},\frac{m}{6}\left\vert \frac{a}{%
m}\right\vert ^{\frac{q(n-k)}{k}}\right) \right) ^{\frac{1}{q}}\right. \\
&&\left. +\left( 2A\left( \frac{1}{6}\left\vert b\right\vert ^{\frac{q(n-k)}{%
k}}+\frac{m}{3}\left\vert \frac{A(a,b)}{m}\right\vert ^{\frac{q(n-k)}{k}%
}\right) \right) ^{\frac{1}{q}}\right] .
\end{eqnarray*}
\end{proposition}

\begin{proof}
The assertion follows from Theorem 8 applied to $f(x)=x^{\frac{n}{k}},$
which\ is an $m-$convex function.
\end{proof}


\begin{thebibliography}{99}
\bibitem{b} S.S. Dragomir, Two mappings in connection to Hadamard's
inequalities, Journal of Math. Anal. Appl., 167 (1992), 49-56.

\bibitem{c} U.S. K\i rmac\i , Inequalities for differentiable mappings and
applications to special means of real numbers to midpoint formula, Appl.
Math. Comput., 147 (2004), 137-146.

\bibitem{d} U.S. K\i rmac\i\ and M.E. \"{O}zdemir, On some inequalities for
differentiable mappings and applications to special means of real numbers
and to midpoint formula, Appl. Math. Comput., 153 (2004), 361-368.

\bibitem{6} C.E.M. Pearce and J. Pe\v{c}ari\'{c}, Inequalities for
differentiable mappings with application to special means and quadrature
formulae, \textit{Appl. Math. Lett.}, 13(2) (2000), 51--55.

\bibitem{a} J.E. Pe\v{c}ari\'{c}, F. Proschan and Y.L. Tong, Convex
Functions, Partial Orderings and Statistical Applications, Academic Press
Inc., 1992.

\bibitem{7} M.W. Alomari, M. Darus and U.S. K\i rmac\i , Some inequalities
of Hermite-Hadamard type for $s-$convex functions, \textit{Acta Mathematica
Scientia,} (2011) 31B(4):1643--1652.

\bibitem{BOP} M.K. Bakula, M.E. \"{O}zdemir and J. Pe\v{c}ari\'{c},
Hadamard-type inequalities for $m-$convex and $(\alpha ,m)-$convex
functions, \textit{J. Inequal. Pure and Appl. Math.}, 9, (4), (2007),
Article 96.

\bibitem{BPR} M.K. Bakula, J. Pe\v{c}ari\'{c} and M. Ribibi\'{c}, Companion
inequalities to Jensen's inequality for $m-$convex and $(\alpha ,m)-$convex
functions, \textit{J. Inequal. Pure and Appl. Math.}, 7 (5) (2006), Article
194.

\bibitem{ST} S.S. Dragomir and G. Toader, Some inequalities for $m-$convex
functions, Studia University Babes Bolyai, \textit{Mathematica}, 38 (1)
(1993), 21-28.

\bibitem{TOA} G. Toader, Some generalisation of the convexity, \textit{Proc.
Colloq. Approx. Opt.}, Cluj-Napoca, (1984), 329-338.

\bibitem{MER} M.E. \"{O}zdemir, M. Avc\i\ and E. Set, On some inequalities
of Hermite-Hadamard type via $m-$convexity, \textit{Applied Mathematics
Letters}, 23 (2010), 1065-1070.

\bibitem{TOA2} G. Toader, On a generalization of the convexity, \textit{%
Mathematica}, 30 (53) (1988), 83-87.

\bibitem{SS3} S.S. Dragomir, On some new inequalities of Hermite-Hadamard
type for $m-$convex functions,\textit{\ Tamkang Journal of Mathematics}, 33
(1) (2002).

\bibitem{2} H. Kavurmaci, M. Avci, M.E. \"{O}zdemir, New Ostrowski type
inequalities for $m-$convex functions and applications, accepted.

\bibitem{3} M.E. \"{O}zdemir, E. Set and M.Z. Sar\i kaya, Some new
Hadamard's type inequalities for co-ordinated $m-$convex and $(\alpha ,m)-$%
convex functions, \textit{Hacettepe J. of. Math. and Ist.}, 40, 219-229,
(2011).

\bibitem{4} M.Z. Sar\i kaya, M.E. \"{O}zdemir and E. Set, Inequalities of
Hermite--Hadamard's type for functions whose derivatives absolute values are 
$m-$convex, \textit{RGMIA Res. Rep. Coll.} 13 (2010) Supplement, Article 5.

\bibitem{8} S.S. Dragomir and R.P. Agarwal, Two inequalities for
differentiable mappings and applications to special means of real numbers
and to trapezoidal formula, \textit{Appl. Math. Lett.,} 11 (5), 91-95,
(1998).

\bibitem{F} N. Latif, J. Pe\v{c}ari\'{c} and I. Peri\'{c}, Some New Results
Related to Favard's Inequality, \textit{J. Inequal. Appl.}, 2009, Article ID
128486.

\bibitem{pp} J. Pe\v{c}ari\'{c}, F. Proschan and Y.L. Tong, Convex
functions, Partial Orderings and Statistical Applications, \textit{Academic
Press}, 1992.
\end{thebibliography}
\end{document}